\documentclass[12pt,centertags,oneside, reqno]{amsart}
\usepackage{ifpdf}

\usepackage{amssymb}
\usepackage{braket}
\usepackage[all]{xy}
\usepackage{cancel}

\usepackage{amscd,amsxtra,calc}
\usepackage{cmmib57}
\usepackage{url}
\usepackage{ulem}
\usepackage{xcolor}
\usepackage[a4paper,width=16.2cm,top=3cm,bottom=3cm]{geometry}

\numberwithin{equation}{section}

\theoremstyle{plain}
\newtheorem{thm}{Theorem}[section]
\newtheorem{theorem}[thm]{Theorem}

\newtheorem{lemma}[thm]{Lemma}

\newtheorem{proposition}[thm]{Proposition}
\theoremstyle{definition}
\newtheorem{remark}[thm]{Remark}

\newtheorem{definition}[thm]{Definition}

\newtheorem{example}[thm]{Example}

\newtheorem{question}[thm]{Question}

\numberwithin{equation}{section}

\newcommand{\sO}{{\mathcal O}}

\newcommand{\C}{{\mathbb C}}

\renewcommand{\P}{{\mathbb P}}
\newcommand{\Q}{{\mathbb Q}}
\newcommand{\R}{{\mathbb R}}

\newcommand{\Z}{{\mathbb Z}}

\newcommand{\id}{{\rm id\hspace{.1ex}}}

\newcommand{\Aut}{{\rm Aut\hspace{.1ex}}}

\title[Fibered Calabi-Yau threefolds]{Fibered Calabi-Yau threefolds with relative automorphisms of positive entropy and $c_2$-contractions} 

\author{Keiji Oguiso}
\address{Mathematical Sciences, the University of Tokyo, Meguro Komaba 3-8-1, Tokyo, Japan, and National Center for Theoretical Sciences, Mathematics Division, National Taiwan University, 
Taipei, Taiwan}
\email{oguiso@ms.u-tokyo.ac.jp}

\thanks{The author is partially supported by JSPS Grant-in-Aid (S) 15H05738, JSPS Grant-in-Aid (B) 15H03611 and NCTS scholar program.}
\subjclass[2010]{14J32, 14J50}

\begin{document}

\maketitle

\begin{abstract}
We show that an abelian fibered Calabi-Yau threefold with a positive entropy automorphism preserving the fibration is unique up to isomorphisms as fibered varieties. We also give a fairly explicit structure theorem of an elliptically fibered Calabi-Yau threefold with a birational automorphism of "positive entropy" preserving the fibration.
\end{abstract}

\section{Introduction}

There are a lot of interesting projective surface automorphisms of positive entropy (see e.g. \cite{Mc16}, \cite{Ue16}, \cite{OY20} and references therein). Target surfaces are necessarily birational to either $\P^2$, K3 surfaces, abelian surfaces or Enriques surfaces by a result of Cantat (\cite{Ca99}). In \cite{Og14}, we discussed some difficulty to construct "interesting" biregular automorphisms of positive entropy of a smooth projective vaeriety of dimension greater than or equal to three. There we considered that primitive automorphisms of positive entropy introduced by Zhang \cite{Zh09} are interesting ones and gave such a Calabi-Yau threefold example and a rational threefold example found first by Truong and me \cite{OT15} (see also \cite{Og19} for further progress). 

\medskip

In this paper, we continue to discuss the same kind of problems but from a completely different view point. We study fibered Calabi-Yau threefolds with equivariant automorphisms of positive entropy, hence, non primitive ones. The aim of this paper is to point out that the existence of such an automorphism gives rather strong constraint of the target Calabi-Yau threefold; the existence of such an automorphism uniquely determines the fibered Calabi-Yau threefold when it is an abelian fibered Calabi-Yau threefold (Theorem \ref{mainthm1}) and gives fairly explict concrete structure of the fibered Calabi-Yau threefolds when they are elliptically fibered threefolds (Theorem \ref{mainthm2}), with various examples related (Examples \ref{ex2}, \ref{ex22}). This is not the case for pluricanocical fibrations and maximally rational connected fibrations (Examples \ref{ex2} and \ref{ex3}).

\medskip

Unless stated otherwise, we work over $\C$. In order to state our main results, we fix a few terminologies first.

We denote by $N^1(V)$ the free $\Z$-module of finite rank generated by the numerically equivalence classes of Cartier divisors on a normal projective variety $V$. Then $\Aut (V)$ acts naturally on $N^1(V)$. We denote $N^1(V) \otimes_{\Z} \Q$ (resp. $N^1(V) \otimes_{\Z} \R$) by $N^1(V)_{\Q}$ (resp. $N^1(V)_{\R}$). 
  
We call a surjective morphism $f : V \to B$ from a normal projective variety $V$ to a normal projective variety $B$ with connected fibers a {\it fiber space}, a fibration, or a contraction. Two fiber spaces $f : V \to B$ and $f' : V' \to B'$ are isomorphic (resp. birationally isomorphic) if there are isomorphisms $\tau_V : V \to V'$ and $\tau_B : B \to B'$ (resp. birational maps $\tau_V : V \dasharrow  V'$ and $\tau_B : B \dasharrow B'$) such that $\tau_B \circ f = f' \circ \tau_V$. 

Let $f : V \to B$ be a fiber space. We denote by $\Aut (V)$ the group of biregular automorphisms of $V$ and define its subgroups $\Aut (f)$ and $\Aut (V/B)$ by
$$\Aut (f) := \{g \in \Aut (V)\,|\, f \circ g = g_B \circ f\,\, \exists g_B \in \Aut (B)\},$$
$$\Aut (V/B) := \{g \in \Aut (V)\,|\, f \circ g = f\} \subset \Aut (f).$$ 
We define ${\rm Bir}\, (f)$ and ${\rm Bir}\, (V/B)$ by replacing $\Aut (V)$ and $\Aut (B)$ above by ${\rm Bir}\, (V)$ and ${\rm Bir}\, (B)$, the groups of birational automorphisms of $V$ and $B$. We call a group $G$ virtually abelian if $G$ has an abelian subgroup of finite index.

An automorphism $g \in \Aut (V)$ of a smooth projective variety $V$ is {\it of positive entropy} if and only if the first dynamical degree $d_1(g)$ of $g$, which is equal to the spectral radius of $g^*|_{N^1(V)}$ when $g \in \Aut (V)$, is strictly great than $1$. Otherwise, we have $d_1(g) =1$ and we say that $g$ is of zero entropy. See the original paper \cite{DS05} and a survey \cite{Og14} and references therein, for the precise definition of entropy, dynamical degrees and their basic properties that we need in this paper. 

\medskip

We call a smooth projective threefold $X$ a {\it Calabi-Yau threefold} if $\sO_X(K_X) \simeq \sO_X$ and $\pi_1(X^{{\rm an}}) = \{1\}$, where $X^{{\rm an}}$ is the underlying complex manifold of $X$. Let $f : X \to B$ be a fibration from a Calabi-Yau threefold $X$. Then smooth fibers of $f$ are either K3 surfaces or abelian surfaces and $B = \P^1$ if $\dim B = 1$, while smooth fibers are elliptic curves and $B$ is a rational surface with at most klt as its singularities if $\dim B = 2$ (\cite{Og93}). We call $f$ an abelian fibration if general fibers of $f$ are abelian surfaces and an elliptic fibration if general fibers of $f$ are elliptic curves. We do not assume the existence of rational section. 

There are many Calabi-Yau threefolds (See e.g. \cite[Part II]{GHJ03}). Let us next recall a very special Calabi-Yau threefold $X_3$ with an abelian fibration, which will play a curcial role in this paper:

\begin{example}\label{ex1} Let $E_{\omega}$ be the elliptic curve of period $\omega := \frac{-1 + \sqrt{-3}}{2}$ and 
$$\pi_3 : X_3 \to \overline{X}_3$$ 
be the blow up at the maximal ideals of the singular points of $\overline{X}_3$ of the quotient variety 
$$\overline{X}_3 := E_{\omega}^3/\langle \omega \cdot {\rm id}_{E_{\omega}^3} \rangle$$
of the product threefold $E_{\omega}^3$. Then $X_3$ is a Calabi-Yau threefold (\cite{Be83}). We regard $X_3$ as an abelian fibered Calabi-Yau threefold by the morphism 
$$\varphi_3 : X_3 \xrightarrow{\pi_3} \overline{X}_3 \xrightarrow{\overline{{\rm pr}}_3} E_{\omega}/\langle \omega \cdot {\rm id}_{E_{\omega}} \rangle = \P^1$$
induced from the third projection ${\rm pr}_3 : E_{\omega}^3 \to E_{\omega}$ and $\pi_3$. Note that $\varphi_3 : X_3 \to \P^1$ has an automorphism $g \in \Aut(\varphi_3)$ of positive entropy. For instance, the automorphism $g \in \Aut (\varphi_3)$ induced from $g_2 \times {\rm id}_{E_{\omega}} \in \Aut ({\rm pr}_3)$, where $g_2 \in \Aut (E_{\omega}^2)$ is given by the matrix
$$\left(
    \begin{array}{cc}
      1 & 1\\
      1 & 0\\
\end{array}
  \right)\,\, ,\,\, {\rm i.e.,}\,\, g_2(P, Q) = (P+Q, P)$$
is such an automorphism. Indeed, one has 
$$d_1(g) = d_1(g_2 \times {\rm id}_{E_{\omega}}) = d_1(g_2) = (\frac{1+\sqrt{5}}{2})^2 >1.$$
Moreover, $\Aut (E_{\omega}^2)$ is isomorphic to a subgroup of $\Aut (\varphi_3)$ via 
$$\Aut (E_{\omega}^2) \ni h \mapsto h \times {\rm id}_{E_{\omega}} \in \Aut (\pi_3),$$
and we have 
$${\rm GL}_2(\Z) \subset \Aut (E_{\omega}^2) \subset \Aut (\pi_3).$$ 
In particular, $\Aut (\varphi_3)$ is not a virtaully abelian group.
\end{example}

Our first main result is the following:

\begin{theorem}\label{mainthm1} Let $f : X \to B$ be an abelian fibered Calabi-Yau threefold. Assume that there is $g \in \Aut (f)$ of positive entropy. Then $f : X \to B$ is isomorphic to $\varphi_3 : X_3 \to \P^1$ in Example \ref{ex1} as fiber spaces. In particular, any $g \in \Aut (f)$ is of zero entropy unless $X \simeq X_3$.
\end{theorem}

It is surprizing that only the existence of a relative automorphism of positive entropy determines the target fibered Calabi-Yau threefold uniquely among all abelian fibered Calabi-Yau threefolds. The following example may be interesting if one compares with Theorem \ref{mainthm1}.

\begin{example}\label{ex2} 
\begin{enumerate}

\item Let us consider two relatively minimal rational elliptic surface $\varphi_i : S_i \to \P^1$ ($i=1$, $2$) defined by the Weierstrass equations 
$$y^2 = x^3 -(t^6-1)\,\, ,\,\, y^2 = x^3 -(t^6-a)\,\, (a \not= 0,1).$$
Then the fiber product 
$$f : V := S_1 \times_{\P^1} S_2 \to \P^1$$ 
over the base space $\P^1$ is a Calabi-Yau threefold with topological Euler number $0$ (\cite{Sc88}, \cite{GLW22}). By construction, $f$ is an abelian fibration whose smooth fibers $V_t$ are isomorphic to $E_{\omega}^2$ as $\varphi_3$ in Example \ref{ex1}. Thus $g_2 \in \Aut (V_t)$ in Example \ref{ex1} for all smooth $V_t$ might seem to be gathered together to be an element $\tilde{g} \in {\rm Bir}\, (V/\P^1)$. {\it However, this is not true!} Note that, since singular fibers $V_s$ ($s^6 =1$ and $a$) of $f$ contain no rigid rational curves, $V$ has no flops over $\P^1$ (\cite{Ka08}). Then if we would obtain $\tilde{g} \in {\rm Bir}\, (V/\P^1)$, then $\tilde{g} \in \Aut\, (V/\P^1)$ and $d_1(\tilde{g}) >1$. Thus $V \simeq X_3$ by Theorem \ref{mainthm1}. However, $V$ is of topological Euler number $0$, while $X_3$ is of topological Euler number $72$ so that they are not isomorphic.

\item $\varphi_3 : X_3 \to \P^1$ has 27 $\P^1$ of normal bundle $\sO_{\P^1}(-1)^{\oplus 2}$ in the singular fibers and we can perform Atiyah flop of these 27 $\P^1$ within the category of projective varieties and get a new abelian fibered Calabi-Yau threefold $\varphi_3' : X_3' \to \P^1$ (\cite[Proposition 2.4]{Og96}). Then $X_3'$ is birational to $X_3$ over $\P^1$, but there is no element of $\Aut (\varphi_3')$ of positive entropy. Indeed, $X_3'$ is not even homeomorphic to $X_3$ by \cite[Theorem 3.1]{LO09} so that the result follows from Theorem \ref{mainthm1}.

\item There are many explicit examples of abelian fibered Calabi-Yau threefolds $f : V \to \P^1$ with $|\Aut (V/\P^1)| = \infty$, which are not isomorphic to $\varphi_3 : X_3 \to \P^1$. See e.g. \cite{Bor91}, \cite{Sc88}, \cite{GLW22}. For such $f$, we have $d_1(h) = 1$ for every $h \in \Aut (f)$ again by Theorem \ref{mainthm1}.

\item Let $f : V \to B$ be a Lagrangian fibration of a projective hyperk\"ahler manifold $X$. Then general fibers of $f$ are abelian varieties of dimension $\dim X/2$ and $\rho(B) = 1$ (\cite{Ma99}). Then $g^*L = L$ for every $g \in \Aut (f)$. Here $L$ is the pullback of the ample generator of $N^1(B)$. We have $q_V(L) = 0$ with respect to the Beauvill-Bogomolov form $q_V$ and therefore $d_1(g) = 1$ and $\Aut (f)$ is a vertually abelian group by \cite[Proposition 2.7]{Og07}. More strongly, by Amerik and Verbitsky \cite[Proposition 3.8 and its proof]{AV23}, the translation automorphisms of fibers in ${\rm Aut}\, (f)$ form a finite index subgroup of ${\rm Aut}\, (f)$ (cf. Example \ref{ex1}). 
 
\item Let $E$ be an elliptic curve and $C$ a smooth projective curve of genus $\ge 2$. Then $V := E \times E \times C$ is of Kodaira dimension $1$ and the third projection ${\rm pr}_3 : V \to C$ is the pluri-canonical morphism of $V$, which is also an abelian fibration on $V$. The automorphism $g_2 \times {\rm id}_C$ of the same shape as in Example \ref{ex1} is an element of $\Aut ({\rm pr}_3)$ and it is of positive entropy. However, $V = E \times E \times C$ form $3g-2$ dimensional family, and $V$ is far from being unique.

\item There are infinitely many smooth projective rational surfaces $S$ with automprphism $g$ of positive entropy (\cite{Mc07}, \cite{BK09} for interesting explicit examples). Let $C$ be a smooth projective curve of genus $\ge 1$. Then the projection ${\rm pr}_2 : S \times C \to C$ is a maximal rationally connected fibration of $S \times C$ with automorphism $(g, {\rm id}_C) \in \Aut ({{\rm pr}_2})$ of positive entropy. Again, such $S \times C$ is far from being unique.
\end{enumerate}
\end{example}

The following question might be interesting (cf. Examples \ref{ex1}, \ref{ex2} (3)(4)):

\begin{question}\label{ques1} Let $f : X \to B$ be an abelian fibered Calabi-Yau threefold which is not isomorphic to $\varphi_3 : X_3 \to \P^1$ in Example \ref{ex1} as fiber spaces. Is $\Aut (f)$ a virtually abelian group?
\end{question}

We call an elliptically fibered Calabi-Yau threefold $f : X \to W$ is {\it of Type $II_0$} if $f$ is a {\it $c_2$-contraction}, that is, $(c_2(X).f^*H)_X = 0$ for an ample class $H \in N^1(W)$ (\cite{OS01}). We have a fairly concrete classification of an elliptically fibered Calabi-Yau threefold of Type $II_0$ (\cite[Theorem 3.4]{OS01}, see also Theorem \ref{thm24} in Subsection \ref{subsect23}). 

Our second main result is the following:

\begin{theorem}\label{mainthm2} Let $f : X \to B$ be an elliptically fibered Calabi-Yau threefold. Assume that there is $g \in {\rm Bir}\, (f)$ with $d_1(g) > 1$ and denote by $g_B \in {\rm Bir}\, (B)$ the birational automorphism of $B$ induced from $g$. Then:
\begin{enumerate}

\item There is an elliptically fibered Calabi-Yau threefold $f' : X' \to B'$ of Type $II_0$, which is birationally isomorphic to $f : X \to B$, such that $g_{B'} \in \Aut (B')$ with $d_1(g_{B'}) > 1$ for $g_{B'} \in {\rm Bir}\, (B')$ induced from $g_B \in {\rm Bir}\, (B)$.

\item Asuume further that $g \in \Aut (f)$ (and thus $g_B \in \Aut (B)$). Then there are a birational morphism $\tau : B \to B'$ and $g_{B'} \in \Aut (B')$ with $d_1(g_{B'}) > 1$ for $g_{B'} \in {\rm Bir}\, (B')$ induced from $g_B \in \Aut (B)$ such that 
$$\tau \circ f : X \to B \to B'$$
is an elliptically fibered Calabi-Yau threefold of Type $II_0$ with $g \in \Aut (\tau \circ f)$.  
\end{enumerate} 
In particular, in both cases, $f : X \to B$ is isotrivial in the sense that any general fibers of $f$ are isomorphic.
\end{theorem}

\begin{remark} The assertion (1) in Theorem \ref{mainthm2}, which is not in the first version, is entirely inspired by a very interesting question asked by the referee.
\end{remark}

There are several examples of elliptically fibered Calabi-Yau threefolds in Theorem \ref{mainthm2} as shown in Example \ref{ex22} in Subsection \ref{subsect23}. Theorem \ref{mainthm2} has again some sharp contrast to the pluri-canonical maps of threefolds of Kodaira dimension $2$ as shown in Example \ref{ex3} below, but in a completely different way from Kodaira dimension $1$ (Example \ref{ex2}):

\begin{example}\label{ex3} Let $X$ be a smooth projective threefold with Kodaira dimension $2$ and $\varphi : X \dasharrow B$ be a pluri-canonical map of $X$. Then $\varphi$ is an elliptic fibration and ${\rm Bir}\, (X)$ acts biregularly on $B$ as a finite group \cite[Theorem 14.10]{Ue75}. Then $d_1(g) = 1$ for all $g \in {\rm Bir}\, (X)$ by the product formula of the dynamical degrees (\cite[Theorem 1.1]{DN11}, see also \cite{Tr20}, \cite{Tr15} for purely algebro-geometric formulation and proof).
\end{example}

We prove Theorem \ref{mainthm1} in Section \ref{sect3} and Theorem \ref{mainthm2} in Section \ref{sect4}, using $c_2$-contraction theorem of Calabi-Yau threefolds \cite{OS01} (see also Theorem \ref{thm24}) with an idea in \cite[Section 5]{KOZ09} and general theory of elliptic threefolds due to Nakayama \cite[Appendix A]{Na02}.

\medskip

It may be interesting to study a K3 fibered Calabi-Yau threefold $f : X \to \P^1$ with relative automorphisms of positive entropy. As the first step, the following question might be interesting:

\begin{question}\label{ques2} Is any K3 fibered Calabi-Yau threefold $f : X \to \P^1$ with relative automorphisms of positive entropy isotrivial in the sense that smooth fibers of $f$ are isomorphic?
\end{question}

This is true for abelian fibered cases and elliptically fibered cases by Theorems \ref{mainthm1} and \ref{mainthm2}.

\medskip

{\bf Acknowledgements.} I would like to thank Professor J\'anos Koll\'ar for his valuable comments in a very beginning stage of this work. I would like to thank Professors Ken Abe, Isamu Iwanari and Sho Tanimoto for their invitation to Kinosaki Algebraic Geometry Symposium 2023. I would like to thank the referee for her/his careful reading with several constructive comments and a very interesting question, all of which are reflected in the revised version. This paper is an expanded version of a report of my talk there (\cite{Og23}, in Japanese). 

\section{Preliminaries}\label{sect2}

In this section, we recall some facts used in the proof of Theorems \ref{mainthm1} and \ref{mainthm2}.

\subsection{First dynamical degrees}\label{subsect21} 
\hfill

Let $V$ be a smooth projective variety and $g \in {\rm Bir}\, (V)$ be a birational automorphism of $V$ such that $g$ is isomorphic in codimension one. Then we have $(f^n)^* = (f^*)^n$ on $N^1(V)$ and $f^*$ preserves the movable cone $\overline{{\rm Mov}}\, (V) \subset N^1(V)_{\R}$. If in addition $f \in \Aut, (V)$, then $f^*$ also preserves the intersection product and the nef cone ${\rm Nef}\, (V) \subset N^1(V)_{\R}$. We frequently use the following well-known consequence of Perron-Frobenius-Birkhoff theorem (\cite{Bi67}, see also \cite[Appendix A]{NZ09}, \cite{Og14}), which claims that {\it the spectral radius of a linear invertible map $F$ of a finite dimensional $\R$-vector space $L$ preserving a strictly convex closed cone $C \subset L$ with non-empty interior, is a real eigenvalue of $F$ with real eigenvector in $C$}:

\begin{theorem}\label{thm21} Let $V$ be a smooth projective variety and $g \in {\rm Bir}\, (V)$ be a birational automorphism of $V$ such that $g$ is isomorphic in codimension one. Then the first dynamical degree $d_1(g)$ of $g$ coincides with the spectral radius of $g^*|_{N^1(V)_{\R}}$ and it is a real eigenvalue of $g^*|_{N^1(V)_{\R}}$. Moreover: 
\begin{enumerate}
\item There is $v \in \overline{{\rm Mov}}\, (V) \setminus \{0\}$ such that $g^*v = d_1(g)v$, that is, there is a real nef eigenvector $v \in \overline{{\rm Mov}}\, (V)$ of $g^*|_{N^1(V)_{\R}}$ with eigenvalue $d_1(g)$. 

\item Assume further that $g \in \Aut (V)$. Then there is $v \in {\rm Nef}\, (V) \setminus \{0\}$ such that $g^*v = d_1(g)v$, that is, there is a real nef eigenvector $v \in {\rm Nef}\, (V)$ of $g^*|_{N^1(V)_{\R}}$ with eigenvalue $d_1(g)$.
\end{enumerate} 
\end{theorem}

We define the first dynamical degree $d_1(g)$ of an automorphism $g$ of a normal projective variety $V$ as the spectral radius of $g^*|_{N^1(V)_{\R}}$, which will be used in the proof of Theorem \ref{mainthm2}. The following comparison is well known and will be also used:

\begin{proposition}\label{prop22} Let $W$ be a normal projective variety of dimension $d$, $g \in {\rm Aut}\, (W)$, $\nu : \tilde{W} \to W$ an equivariant projective resolution with respect to $g$ and $\tilde{g} \in \Aut (\tilde{W})$ the induced automorphism of $\tilde{W}$. 
Then $d_1(\tilde{g}) = d_1(g)$. Moreover, they coincide with the spectral radii of $\tilde{g}^*|_{N^1(\tilde{W})_{\R}}$ and $g^*|_{N^1(W)_{\R}}$ with an eigenvector in ${\rm Nef}\,(\tilde{W})$ and in ${\rm Nef}\,(W)$ respectively.
\end{proposition}

\begin{proof} Let $H$ be a very ample Cartier class on $W$ and $B = \nu^{*}H$, which is a nef and big class on $\tilde{W}$. Let $\tilde{H}$ be a very ample divisor on $\tilde{W}$. Then there are positive integers $m$ and $m'$ and effective classes $E$ and $E'$ on $\tilde{W}$ such that 
$$\tilde{H} + E = mB\,\, ,\,\, B+ E' = m'\tilde{H}\,\, {\rm in}\,\, N^1(\tilde{W}).$$
Combining this with $((\tilde{g}^n)^*B.B^{d-1})_{\tilde{W}} = ((g^n)^*H.H^{d-1})_{W}$, it readily follows that $d_1(\tilde{g}) = d_1(g)$. The second statement follows from \cite[Proposition A.2]{NZ09}, as $(g^n)^* = (g^{*})^n$ for $g \in {\rm Aut}\, (W)$ and therefore we may apply it for the action of $g^*$ on ${\rm Nef}\, (W) \subset N^1(W)_{\R}$ with ample class $H$ which is in the interior of ${\rm Nef}\, (W)$, and similarly for $\tilde{g}$. 
\end{proof} 

The following proposition is also well known for the experts. Since this proposition is crucial in our proof of Theorem \ref{mainthm2}, following a request by the referee, we also recall a general statement and its proof based on \cite{DS05} and \cite{DN11} (See also \cite{Tr15}, \cite{Tr20} for purely algebro-geometric statements and proofs for the facts we will use in the proof below).  

\begin{proposition}\label{prop24} Let $V$ and $W$ be smooth projective 
varieties such that $\dim V - \dim W =1$ and $f : V \dasharrow W$ be a dominant rational map. Let $g \in {\rm Bir}\, (V)$ and $g_W  \in {\rm Bir}\, (W)$ such that $f \circ g = g_W \circ f$. Then $d_1(g) = d_1(g_W)$.
\end{proposition}

\begin{proof} Let $d = \dim V$ and $e = \dim W$. Then $d-e=1$. Let $p_n : V_n \rightarrow V$ be a resolution of indeterminacy of $g^n$ and $q_n : V_n \to V$ be the induced morphism such that $g^n \circ p_n = q_n$. We define 
$$((g^n)^*D.E^{d-1})_{V} := (q_n^*D.(p_n^*E)^{d-1})_{V_n} = (D.(q_n)_*((p_n^*E)^{d-1}))_{V},$$
where $D$ and $E$ are line bundles on $V$. Then by \cite{DS05}, we have by definition
$$d_0(g) = \lim_{n \to \infty} (H^{d})_V^{\frac{1}{n}} =1\,\, ,\,\, d_1(g) = \lim_{n \to \infty} ((g^n)^*H.H^{d-1})_V^{\frac{1}{n}},$$
where $H$ is a very ample line bundle on $V$. 
In \cite{DS05}, Dinh and Sibony show that the limit $d_1(g)$ exists, hence $d_1(g) \ge 1$, and $d_1(g)$ does not depend on the choices of $H$ and $p_n : V_n \to V$. Moreover, $d_1(g)$ is a birational invariant in the sense that if $\nu : V' \dasharrow V$ is a birational map from a smooth projective variety $V'$ to $V$, then $d_1(g') = d_1(g)$ for $g' := \nu^{-1} \circ g \circ \nu \in {\rm Bir}\, (V')$ (\cite[Corollaire 7]{DS05}). $d_1(g_W)$ ($k=0$ and $1$) is defined in the same way as $d_k(g)$. 

Let $\nu : V' \to V$ be a resolution of indeterminacy of $f$. Set $f':= f \circ \nu$ and $g' = \nu^{-1} \circ g \circ \nu$. Then $f' : V' \to W$ is a surjective morphism and $g' \in {\rm Bir}\, (V')$ satisfies $f' \circ g' = g_W \circ f'$ and $d_1(g') = d_1(g)$ as remarked above. 

{\it So, in what follows, we may and will assume that $f$ is a surjective morphism} (by replacing $f : V \dasharrow W$ and $g$ by $f' : V' \to W$ and $g'$). 

Let $H_W$ be a very ample line bundle on $W$. Then by \cite{DN11} in the special case where $\dim X - \dim W = 1$, we have the relative $k$-th dynamical degree ($k=0$ and $k=1$) for $f : V \to W$: 
$$d_0(g|f) := \lim_{n \to \infty} (H.(f^*H_W)^{d-1})_{V}^{\frac{1}{n}} = 1\,\, ,\,\, d_1(g|f) := \lim_{n \to \infty} 
((g^n)^*H.(f^*H_W)^{d-1})_{V}^{\frac{1}{n}}.$$
Then by the main result of \cite{DN11} applied for $f : V \to W$ with $d-e =1$, we have
$$d_1(f) = {\rm Max}\, (d_0(g_W)d_1(g|f), d_1(g_W)d_0(g|f)).$$
Here $d_0(g_W) = 1$ and $d_0(g|f) =1$. We are going to show that $d_1(g|f) = 1$. Note that 
$$((g^n)^*H.(f^*H_W)^{d-1})_{V} = (q_n^*H.(p_n^*f^*H_W)^{d-1})_{V_n} = (H.(q_n)_{*}((f \circ p_n)^*H_W)^{d-1}))_V.$$
Here, as an element of $N_1(V_n)$ (which is the abelian group consisting of the numerical equivalence classes of one cycles on $V_n$), the one cycle $((f \circ p_n)^*H_W)^{d-1}$ is in the same class as $N\tilde{F}_n$, that is, 
$$((f \circ p_n)^*H_W)^{d-1} = N\tilde{F}_n$$ 
in $N_1(V_n)$, where $N := (H_W)_W^{d-1}$ and $\tilde{F}_n$ is a {\it general} fiber of $f \circ q_n$. 

Here the precise definition of a general fiber $\tilde{F}_n$ of $f \circ p_n$ is as follows. We choose Zariski open dense subsets $U_n \subset W$ and $U_n' \subset W$ such that $g_W^n|_{U_n} : U_n \to U_n'$ is an isomorphism, $f \circ p_n$ (resp. $f \circ q_n$) is smooth over $U_n$ (resp. $U_n'$) and $f$ is smooth over $U_n'$. Such $U_n$ and $U_n'$ exist, as we may first choose $U_n$ and $U_n'$ so that $g_W^n|_{U_n} : U_n \to U_n'$ is an isomorphism. Then we may just shrink $U_n$ suitably to get desired $U_n$ and $U_n' = g_W^n(U_n)$ by the generic smoothness for $g_W^n \circ f \circ p_n = f \circ q_n$ and $f$. Then 
$$\tilde{F}_n := (f \circ p_n)^{-1}(P_n) = (f \circ q_n)^{-1}(g_W^n(P_n)) = q_n^{-1} \circ f^{-1} (g_W^n(P_n))$$
for $P_n \in U_n$. Note that the class $\tilde{F}_n \in N_1(V_n)$ is independent of the choice of the point $P_n$ in $U_n$ as $f \circ p_n$ is smooth (hence flat) over $U_n$. Since $q_n$ is a birational morphism and $f$ is smooth over $U_n$, it follows that 
$$(q_n)_*(((f \circ p_n)^*H_W)^{d-1}) = (q_n)_*(N\tilde{F}_n) = (q_n)_{*}(N q_{n}^{-1}f^{-1}(g_W^n(P_n))) = Nf^{-1}(g_W^n(P_n)) = N F_n$$ 
in $N_1(V)$ for $F_n := f^{-1}(g_W^n(P_n))$. Note that, for each $n$, the number $(H.F_n)_V$ does not depend on the choice of $P_n \in U_n$ as $f$ is smooth (hence flat) over $U_n' = g_W^n(U_n)$. On the other hand, for any positive integer $n$, since the set $U_n \cap U_1$ is Zariski dense in $W$, we may choose $P_n = P_1 \in U_n \cap U_1$. For this choice, we have $F_n = F_1$ and thus $(H.F_n)_{V} = (H.F_1)_{V}$. Hence $(H.F_n)_{V}$ does not depend on the choices of $P_n \in U_n$ and $n$. We denote this number by $M$. Then
$$d_1(g|f) = \lim_{n \to \infty} (N(H.F_n)_{V})^{\frac{1}{n}} = \lim_{n \to \infty} (N.M)^{\frac{1}{n}} = 1,$$
as $N.M > 0$ is constant with respect to $n$.

Hence, by $d_1(g_W) \ge 1$, we obtain 
$$d_1(f) = {\rm Max}\, (d_0(g_W)d_1(g|f), d_1(g_W)d_0(g|f)) = {\rm Max}\, (1, d_1(g_W)) = d_1(g_W),$$
as claimed.
\end{proof}

\subsection{The second Chern class of a Calabi-Yau threefold.}\label{subsect22}
\hfill

Let $X$ be a Calabi-Yau threefold. We regard the second Chern class $c_2(X)$ as a linear form on $N^1(X)_{\R} := N^1(X) \otimes_{\Z} {\R}$ via intersection form $(c_2(X), *)_X$. The hyperplane $(c_2(X), *)_X = 0$ in $N^1(X)_{\R}$ is a rational hyperplane with respect to the integral structure $N^1(X)$. By a famous theorem due to Miyaoka-Yau (\cite{Mi87}), together with the fact that $X$ is not an \'etale quotient of an abelian threefold by $\pi_1(X^{{\rm an}}) = \{1\}$, we have

\begin{theorem}\label{thm23} Let $X$ be a Calabi-Yau threefold. Then, the linear form 
$$(c_2(X). *)_X : N^1(X)_{\R} \to \R\,\, ;\,\, x \mapsto (c_2(X).x)_X$$
is strictly positive on the ample cone ${\rm Amp}\, (X)$ and non-negative on the nef cone ${\rm Nef}\, (X)$, which is the closure of ${\rm Amp}\, (X)$ in $N^1(X)_{\R}$. 
\end{theorem}

The following proposition is due Wilson \cite{Wi97}. This proposition indicates some crucial roles of $c_2(X)$ in the study of $\Aut(X)$ (See also \cite{OP98} for other roles of $c_2(X)$): 

\begin{proposition}\label{prop23} Let $X$ be a Calabi-Yau threefold such that $\Aut (X)$ is an infinite group. Then there is a non-zero real nef class $x \in \partial {\rm Nef}\, (X) \setminus \{0\}$ such that $(c_2(X).x)_X = 0$.
\end{proposition}
Though we will not use this proposition in our proof of main theorems, this proposition indicates some reason why $c_2$-contractions appear in our study.

\subsection{$c_2$-contraction on a Calabi-Yau threefold.}\label{subsect23}
\hfill

Let $f : X \to W$ be a fibration on a Calabi-Yau threefold $X$. 
We call $f : X \to W$ a {\it $c_2$-contraction} if the linear form $(c_2(X).*)_X$ is identically zero on $f^*N^1(W)_{\R}$. This condition is equivalent to the condition that $(c_2(X).f^*H)_X = 0$ for an ample divisor $H$ on $W$ by Theorem \ref{thm23}. We call a $c_2$-contraction $f_0 : X \to W_0$ {\it maximal} if for any $c_2$-contraction $f : X \to W$, there is $\nu_0 : W_0 \to W$ such that $f = \nu_0 \circ f_0$. For any $X$, we have a maximal $c_2$-contraction $f_0 : X \to W_0$, possibly $\dim W_0 = 0$, and a maximal $c_2$-contraction is unique in the sense that if $f_0' : X \to W_0'$ is also a maximal $c_2$-contraction, then there is an isomorphism $\tau_0 : W_0 \to W_0'$ such that $f_0' = \tau_0 \circ f_0$ (\cite[Lemma-Definition 4.1]{OS01}). 

The next proposition will be used in the proof of Theorem \ref{mainthm2}.

\begin{proposition}\label{prop23} Let $f : X \to B$ be an elliptically fibered Calabi-Yau threefold. Then there is an effective $\Q$-divisor $\Delta$ on $B$ such that $(B, \Delta)$ is a klt pair such that $12(K_B + \Delta)$ is linearly equivalent to $0$. In particular $B$ has at most quotient singularities and $\Q$-factorial. Moreover, $f$ of Type $II_0$ if and only if $\Delta = 0$ as a Weil divisor if and only if $K_W$ is $\Q$-linearly equivalent to $0$ if and only if $K_W$ is numerically equivalent to $0$.
\end{proposition}

\begin{proof} See \cite[Lemma 3.5]{Og93} and its proof. The last equivalence follows from the proof of \cite[Lemma 3.5]{Og93}. It also follows from the log abundance theorem for surfaces with klt singularities or a general result due to Gongyo \cite[Theorem 1.2]{Go13}.
\end{proof} 

In order to state $c_2$-contraction theorem (Theorem \ref{thm24}), it is convenient to use the notion of $G$-Hilbert scheme (\cite{IN96}) for 
some special case:

\begin{definition}\label{def21} Let $V$ be a smooth projective threefold such that $H^0(\Omega_V^3) = \C \omega_V$ and the $3$-form $\omega_V$ is nowhere vanishing, and $G$ a finite subgroup of $\Aut (V)$ such that $G$ acts on $H^0(\Omega_V^3)$ as identity. $G$-Hilbert scheme $G-{{\rm Hilb}}\,(V)$ of $V$ is a closed subscheme of ${\rm Hilb}^{|G|}\,(V)$ consisting of $0$-dimensional $G$-invariant closed subscheme $Z$ of $V$ such that the induced representation of $G$ on $H^0(\sO_Z)$ is isomorphic to the regular representation of $G$. We have a natural morphism $\pi : G-{{\rm Hilb}}\,(V) \to V/G$ induced from the Hilbert-Chow morphism ${\rm Hilb}^{|G|}\,(V) \to {\rm Sym}^{|G|}(V) = V/G$. The morphism $\pi$ is a projective crepant resolution of $V/G$ by Bridgeland, King and Reid \cite{BKR01}.
\end{definition}

The following example and theorem are obtained by \cite[Theorem 3.4]{OS01} together with the fact due to \cite{BKR01} in Definition \ref{def21}:

\begin{example}\label{ex21}
\begin{enumerate}
\item Let $\overline{S}$ be a normal projective K3 surface or an abelian surface, $\nu : S \to \overline{S}$ the minimal resolution (possibly identity), $E$ an elliptic curve and $G$ a finite group with faithful actions on both $\overline{S}$ and $E$ such that the induced diagonal action 
$$G \subset \Aut (S) \times \Aut (E) \subset \Aut (S \times E)$$
acts as identity on $H^0(\Omega_{S \times E}^3)$ and the quotient surface $W := \overline{S}/G$ is a rational surface with numerically trivial canonical class. Then $G-{{\rm Hilb}}\,(S \times E)$ is a Calabi-Yau threefold provided that it is simply-connected and the morphism  
$$\varphi : {G}-{{\rm Hilb}}\,(S \times E) \xrightarrow{\pi} (S \times E)/G \xrightarrow{\overline{{\rm pr}_1}} S/G  \xrightarrow{\overline{\nu}} W = \overline{S}/G$$ 
is then an elliptically fibered Calabi-Yau threefold of Type $II_0$. Conversely, for any elliptically fibered Calabi-Yau threefold of Type $II_0$ is obtained by a finite composition of flops of $\varphi : {G}-{{\rm Hilb}}\,(S \times E) \to W$ over $W$ for some $\overline{S}$, $E$, $G$ above. 

\item Let $A_3 := E_{\omega}^3$ and $\mu_3 = \langle \omega \cdot {\rm id}_{A_3} \rangle$, Then $X_3 := \mu_3-{{\rm Hilb}}\, (A_3)$ is a Calabi-Yau threefold and the natural morphism
$$\pi_3 : X_3 = \mu_3-{{\rm Hilb}}\, (A_3) \to \overline{X}_3 := A_3/\mu_3$$
is a $c_2$-contraction on $X_3$.

\item Let $\zeta_7 := {\rm exp}(2\pi\sqrt{-1}/7)$, $A_7$ the Albanese variety of the Klein quartic curve
$$C := \{[x:y:z]\,|\, xy^3+yz^3+zx^3= 0\} \subset \P^2,$$
and $g_7 \in {\rm Aut}\, (A_7)$ the automorphism of $A_7$ induced from $g \in \Aut\, (C)$ defined by 
$$g([x : y : z]) = [\zeta_7 x: \zeta_7^2 y :\zeta_7^4 z].$$ 
Then $\mu_7 := \langle g_7 \rangle$ is a cyclic group of order $7$ and acts on $H^0(\Omega_{A_7}^3)$ as identity. Then $X_7 := \mu_7-{{\rm Hilb}}\, (A_7)$ is a Calabi-Yau threefold and the natural morphism
$$\pi_7 : X_7 := \mu_7-{{\rm Hilb}}\, (A_7) \to \overline{X}_7 := A_7/\mu_7$$ 
is a $c_2$-contraction of a Calabi-Yau threefold $X_7$. 
\end{enumerate}
\end{example}

\begin{theorem}\label{thm24} Let $f : X \to W$ be a $c_2$-contraction of a Calabi-Yau threefold $X$. Then:
\begin{enumerate}
\item Assume that $\dim W = 1$. Then $f : X \to W$ is an abelian fibered Calabi-Yau threefold and $W = \P^1$ and vice versa.
\item Assume that $\dim W = 2$. Then $W = \overline{S}/G$ and $f : X \to W$ is obtained by a composition of flops from $\varphi : G-{\rm Hilb}(S \times E) \to W$ over $W$ for some $S \to \overline{S}$, $E$ and $G$ in Example \ref{ex21} (1) with $\pi_1(G-{\rm Hilb}(S \times E)) = \{1\}$ and vice versa. 
\item Assume that $\dim W = 3$. Then $f : X \to W$ is isomorphic to exactly one of $\pi_3 : X_3 \to \overline{X}_3$ and $\pi_7 : X_7 \to \overline{X}_7$ in Example \ref{ex21} (2) and (3). 
\item $f : X \to W$ in (3) is a maximal $c_2$-contraction in each case. 
\end{enumerate}
\end{theorem} 

\begin{remark} $\pi_3 : X_3 \to \overline{X}_3$ in Example \ref{ex1} and in Example \ref{ex21} coincide by Theorem \ref{thm24} (4). For the same reason, an explicit description of $\pi_7 : X_7 \to \overline{X}_7$ in \cite {OS01} and in Example \ref{ex22} also coincide.  
\end{remark} 

We close this subsection by giving some examples of elliptically fibered Calabi-Yau threefolds of Type $II_0$ with a relative automorphism of postive entropy.

\begin{example}\label{ex22} 
\begin{enumerate}
\item Let $X_3$, $\pi_3 : X_3 \to \overline{X}_3$ and $g \in \Aut (X_3)$ be as in Example \ref{ex1}. Consider the following elliptic fibration
$$\pi_{12} : X_3 \xrightarrow{\pi_3} \overline{X}_3 
\xrightarrow{\overline{{\rm pr}}_{12}} W := E_{\omega}^2/\langle \omega \cdot {\rm id}_{E_{\omega}^2} \rangle $$
induced from the projection ${\rm pr}_{12} : E_{\omega}^3 \to E_{\omega}^2$ to the first two factors and $\pi_3$. Then the fibration $\pi_{12}$ is an elliptically fibered Calabi-Yau threefold of Type $II_{0}$, where $S$ in Theorem \ref{thm24} is $E_{\omega}^2$. Moreover, $g \in \Aut(\pi_{12})$ and $g$ is of positive entropy. 

\item Let $S$ be a projective K3 surface with an involution $\iota_S$ such that $\iota_S^*|_{N^1(S)} = {\rm id}_{N^1(S)}$, $\iota_S^{*}\omega_S = -\omega_S$ and the fixed locus $S^{\iota_S}$ is non-empty and consists of finite disjoint uninon of $\P^1$. Then $\iota_S$ is in the center of $\Aut\, (S)$ and there are 11 families of such $S$ according to $S^{\iota_S}$ and parity (\cite[Theorem 4.2.2, Table 1]{Ni83}). In each case $\Aut (S)$ contains an element $g$ of positive entropy (\cite[Remark 3.9]{Yu23}). Let $\mu : S \to \overline{S}$ be the contraction of $S^{\iota}$ and $B = \overline{S}/\iota$. We define the action of $\mu_2 := \langle \iota \rangle$ on $S \times E$ by $\iota_S \times (-1_{E})$. Then $B = \overline{S}/\mu_2$ is a rational surface such that ${\sO_B}(2K_B) \simeq \sO_B$ and the natural morphism
$$\pi_{1} : X := \mu_2-{\rm Hilb}\, (S \times E) \xrightarrow{\pi} (S \times E)/\langle \iota \rangle \xrightarrow{\overline{{\rm pr}}_{1}} S/\mu_2 \xrightarrow{\overline{\mu}} W := \overline{S}/\mu_2 $$
is an elliptic fibration of Type $II_{0}$ on a Calabi-Yau threfold $X$. The automorphism $g_X \in \Aut (\pi_1)$ induced from $g \times {\rm id}_E \in \Aut(S \times E)$ satisfies $d_1(g_X) = d_1(g) >1$.
\end{enumerate}
\end{example}

\section{Proof of Theorem \ref{mainthm1}.} \label{sect3}
\hfill

In this section, we prove Theorem \ref{mainthm1}. 

Let $f : X \to B = \P^1$ be an abelian fibered Calabi-Yau threefold and assume that there is $g \in \Aut (f)$ such that $d_1(g) >1$. 

\begin{lemma}\label{lem31} 
There is $g_1 \in \Aut (X/\P^1)$ such that $d_1(g_1) > 1$. 
\end{lemma}
\begin{proof} Since $X$ is of non-negative Kodaira dimension, $f$ has at least three singular fibers by Viehweg-Zuo \cite{VZ01}. Let $N \subset B = \P^1$ be the set of critical values of $f$. Then the automorphism $g_B$ 
of $B$ induced from $g \in \Aut (f)$ acts on $N$. Then, for $n := |N|!$, the induced automorphism $g^n|_B = g_B^n$ acts on the set $N$ as identity. Since $|N| \ge 3$ and $B= \P^1$, it follows that $g^n|_B = \id_B$. Hence $g^n \in \Aut (X/\P^1)$. Moreover $d_1(g^n) = d_1(g)^n >1$. Hence we may take $g_1 = g^n$.
\end{proof}

By Lemma \ref{lem31}, we may and will assume that there is $g \in \Aut (X/\P^1)$ such that $d_1(g) > 1$. Let $H$ be a very ample smooth divisor on $X$ and let $A$ be a general smooth fiber of $f$. Then $A$ is an abelian surface and we have $g|_A \in \Aut (A)$ as $g \in \Aut (X/\P^1)$.

\begin{lemma}\label{lem32} 
There is a non-zero nef real divisor $D_1$ on $X$ such that $g^*D_1 = d_1(g)D_1$ in $N^1(X)_{\R}$. This $D_1$ satisfies the following intersection properties:
$$(c_2(X).D_1)_X = 0\,\, ,\,\,(D_1.A.H)_X > 0.$$
\end{lemma}
\begin{proof} By Theorem \ref{thm21}, there is a non-zero nef real divisor $D_1$ on $X$ such that $g^*D_1 = d_1(g)D_1$ in $N^1(X)_{\R}$. Then 
$$(c_2(X).D_1)_X = (g_*c_2(X).D_1)_X = (c_2(X).g^*D_1)_X = d_1(g)(D_1.c_2(X))_X.$$
Thus $(c_2(X).D_1)_X = 0$ by $d_1(g) > 1$. 

Consider the nef real divisor classes $D_1|_H$ and $A|_H$ on a smooth projective surface $H$. If $(D_1.A.H)_X = 0$, then $(D_1|_{H}.A|_{H})_H = 0$ on $H$. Then $D_1|_{H}$ and $A|_{H}$ are proportional in $N^1(H)_{\R}$ by the Hodge index theorem. Then $D_1$ and $A$ are proportional in $N^1(X)_{\R}$ by the Lefschetz hyperplane section theorem (as $\dim X \ge 3$ and smooth). However, this is impossible as $D_1$ and $A$ are eigenvectors of $g^*|_{N^1(X)_{\R}}$ of different eigenvalues $d_1(g) \not= 1$. Thus $(D_1.A.H)_X \not= 0$. Hence $(D_1.A.H)_X > 0$ as it is the intersection numbers of nef divisors.
\end{proof}
 
\begin{lemma}\label{lem33} 
$d_1(g^{-1}) >1$ and there is a non-zero nef real divisor $D_2$ on $X$ such that $g^*D_2 = d_1(g)^{-1}D_2$ in $N^1(X)_{\R}$ and this $D_2$ satisfies the following intersection 
properties:
$$(c_2(X).D_2)_X = 0\,\, ,\,\, (D_2.A.H)_X > 0.$$
\end{lemma}
\begin{proof} Since $N^1(X)$ is a free $\Z$-module of finite rank, it follows that $|{\rm det}\, g^*| = 1$. Since $d_1(g)^{-1}$ is an eigenvalue of $(g^{-1})^*$ with $0 < d_1(g)^{-1} <1$, it follows that $(g^{-1})^*$ has an eigenvalue of modulus $>1$. Thus $d_1(g^{-1}) > 1$. Now applying Lemma \ref{lem32} for $g^{-1}$, we obtain a non-zero real nef divisor $D_2$ with desired properties. 
\end{proof}

Set $M := A+D_1+D_2$ in $N^1(X)_{\R}$. We call $x \in N^1(X)_{\R}$ big, if it is in the big cone ${\rm Big}\, (X)$, which is an open cone in $N^1(X)_{\R}$ generated by the classes of big divisors (\cite[Lemma 1.7]{Ka88}).

\begin{lemma}\label{lem34} 
$M$ is a real nef and big class with $(c_2(X).M)_X = 0$. 
\end{lemma}

\begin{proof} Since $A$, $D_1$ and $D_2$ are nef, so is $M$. Since 
$$(c_2(X).A)_X = c_2(A) = 0\,\, , \,\,(c_2(X).D_1)_X = (c_2(X).D_2)_X= 0$$
we have $(c_2(X).M)_X = 0$ as well. Again, since $A$, $D_1$ and $D_2$ are nef, 
$$(M^3)_X = ((A+D_1+D_2)^3)_X \ge (A.D_1.D_2)_X = (D_1|_A.D_2|_A)_A.$$
Note that $D_1|_A$ and $D_2|_A$ are non-zero nef eigenvectors of $g^*|_A$ on $N^1(A)_{\R}$ of different eigenvalues $d_1(g|_A) >1$ and $d_1(g|_A)^{-1} <1$ by Lemmas \ref{lem32} and \ref{lem33}. Thus, the nef vectors $D_1|_A$ and $D_2|_A$ are not proportional and therefore $(D_1|_A.D_2|_A)_A > 0$ by the Hodge index theorem on $A$. Hence $(M^3)_X > 0$. Thus $M$ is big by \cite[Lemma 2.3]{FM23}, even if $M$ is not rational.
\end{proof}

\begin{lemma}\label{lem35} 
There is a nef and big Cartier divisor $D$ on $X$ such that $(c_2(X).D)_X = 0$. \end{lemma}
\begin{proof} Since $K_X = 0$, by a result of Kawamata \cite[Theorem 5.7]{Ka88}, ${\rm Nef}\, (X) \cap {\rm Big}\, (X)$ is a locally rational polyhedral cone in ${\rm Big}\, (X)$. Therefore, there are an open ball $U$ centered at $M$ in ${\rm Big}\, (X)$ and a unique minimal dimensional face $F$ of ${\rm Nef}\, (X) \cap {\rm Big}\, (X)$, containing $M$ as its interior, such that the face $F$ is the intersection of rational hyperplanes in $U$. In particular, $F \cap U \cap N^1(X)_{\Q}$, where $N^1(X)_{\Q} = N^1(X) \otimes_{\Z} \Q$, is dense in $F \cap U$. Since $(c_2(X).*)_X = 0$ is a rational hyperplane such that $(c_2(X).M)_X = 0$ and $(c_2(X).*)_X \ge 0$ on ${\rm Nef}\, (X)$ by Theorem \ref{thm23}, it follows that $F \cap U$ is contained in the hyperplane $(c_2(X).*)_X = 0$ by the minimality. Therefore, there is a rational divisor $D'$ whose class is in $F \cap U \cap N^1(X)_{\Q}$. Then $D'$ is a rational nef and big divisor with $(c_2(X).D')_X = 0$. Multiplying $D'$ by some positive integer, we obtain a desired divisor $D$. 
\end{proof}

The next lemma completes the proof of Theorem \ref{mainthm1}.
\begin{lemma}\label{lem36} 
$f : X \to B$ is isomorphic to $\pi_3 : X_3 \to \P^1$ as fiber spaces. 
\end{lemma}
\begin{proof}
Since $D$ in Lemma \ref{lem35} is semi-ample as $D = K_X +D$, the linear system $|mD|$ with some large divisible $m$ defines a birational contraction $\pi : X \to \overline{X}$ with $mD = \nu^*H$ for some ample divisor on $\overline{X}$. Then, we have 
$$(c_2(X).\nu^*H)_X = (c_2(X).mD)_X =0.$$
Therefore $\nu : X \to \overline{X}$ is a birational $c_2$-contraction. Thus by Theorem \ref{thm24}, $\pi : X \to \overline{X}$ is isomorphic to either $\pi_3 : X_3 \to \overline{X}_3$ or $\pi_7 : X_7 \to \overline{X}_7$ under the notation in Theorem \ref{thm24}. By the maximality in Theorem \ref{thm24} (4), the morphism $\nu : X \to \overline{X}$ factors through $f : X \to B$, as an abelian fibration is a $c_2$-contraction. Since $\overline{X}_7$ admits no non-trivial fibration \cite[Lemma 4.2]{OS01}, it follows that $X$ is isomorphic to $X_3$ and $f : X \to B$ is isomorphic to an abelian fibration on $X_3$. Again by \cite[Proposition 3.2 (2), see also Page 73]{OS01} and by the fact that any abelian fibration of $X_3$ factors through $\pi_3 : X_3 \to \overline{X}_3$ by the maximality, any abelian fibration on $X_3$ is isomorphic to $\varphi_3 : X_3 \to \P^1$ as fibered varieties. Hence $f : X \to B$ is isomorphic to $\varphi_3 : X_3 \to \P^1$ as fibered varieties. 
\end{proof}

This completes the proof of Theorem \ref{mainthm1}.

\section{Proof of Theorem \ref{mainthm2}.}\label{sect4}

In this section, we prove Theorem \ref{mainthm2}. Let $f : X \to B$ be an elliptically fibered Calabi-Yau threefold with $g \in {\rm Bir}\, (f)$ such that $d_1(g) > 1$. Since $X$ is a smooth projective variety with nef canonical divisor (hence is minimal), $g$ is isomorphic in codimension one (\cite{Ka08}). Let $g_B \in {\rm Bir}\, (B)$ be a birational automorphism of $B$ induced from $g \in {\rm Bir}\, (f)$.

\begin{proposition}\label{prop41} There is an elliptically fibered Calabi-Yau threefold $f' : X' \to W$ birationally isomorphic to $f : X \to B$ such that 
$g_W \in \Aut (W)$
for $g_W \in {\rm Bir}\, (W)$ induced from $g_B \in {\rm Bir}\, (W)$. 
Moreover, $d_1(g_W) > 1$. If $g_B \in \Aut (B)$, then $d_1(g_B) >1$. 
\end{proposition}

\begin{proof} Let $f' : X' \to W$ be a standard model of $f : X \to B$ in the sense of Nakayama \cite[Definition A.2]{Na02}. Since $X$ is smooth projective, $K_X$ is trivial (hence nef) and $B$ is also projective, we may perform the same construction as in \cite[Appendix A]{Na02} (in which everything is formulated in the category of analytic varieties with dimension $\le 3$) just by running a usual minimal model program in the category of complex projective varieties of dimension $\le 3$. Thus, in the definition of a standard model in \cite[Definition A.2]{Na02}, all locally projective morphisms are chosen to be projective morphisms, a standard model $f' : X' \to W$ exists in the category of projective varieties, $X'$ is a minimal projective variety birational to $X$ and $W$ is a normal projective surface, for our $f : X \to B$. Since $X'$ and $X$ are both minimal and they are birational, any birational map $X \dasharrow X'$ is isomorphic in codimension one and is decomposed into finitely many flops (\cite{Ka08}). Since $X$ is a smooth Calabi-Yau threefold, $X'$ is then a smooth Calabi-Yau threefold, as flops of threefolds preserves the smoothness \cite[Theorem 2.4]{Ko89}, as well as triviality of the canonical divisor and the fundamental group (as a flop is a birational map which is isomorphic in codimension one). 

Let $g' \in {\rm Bir}\, (X')$ be the birational automorphism of $X'$ induced from $g \in {\rm Bir}\, (X)$. Since $X'$ is a Calabi-Yau threefold, for the same reason as above, birational maps
$$g' : X' \dasharrow X'\,\, {\rm and}\,\, (g')^{-1} : X' \dasharrow X'$$ 
are isomorphic in codimension one, that is, none of $g'$ and $(g')^{-1}$ contracts a divisor to a curve or a point. Since $f'$ is equidimensional, it follows that none of $g_{W}$ and $g_{W}^{-1}$, the birational automorphisms of $W$ induced from $g'$ and $(g')^{-1}$, contracts a prime divisor to a point, either. Thus $g_W : W \dasharrow W$ is also isomorphic in codimension one. Since $W$ is a normal projective surface, it follows that $g_W : W \dasharrow W$ is then an isomorphism by the Zariski main theorem. Hence $g_W \in \Aut (W)$ as claimed.

Let us show that $d_1(g_W) > 1$. Let $\nu : \tilde{W} \to W$ be an equivariant resolution 
with respect to $g_W$ and $\tilde{g}_{W} \in \Aut (\tilde{W})$ the induced automorphism. Let $\tilde{f} : X \dasharrow \tilde{W}$ be the dominant rational map induced from $f$ and $\nu$. Then 
$$d_1(g_W) = d_1(\tilde{g}_W)$$ 
by Proposition \ref{prop22}. Since $\dim X - \dim W = 1$, it follows from Proposition \ref{prop24} that 
$$d_1(\tilde{g}_W) = d_1(g) > 1.$$ 
Hence $d_1(g_W) = d_1(g) > 1$ as claimed.

If $g_{B} \in \Aut (B)$, then we may apply the same argument for an equivariant resolution $\tilde{B} \to B$
with respect to $g_B$ and $\tilde{g}_{B} \in \Aut (\tilde{B})$ to obtain $d_1(g_B) = d_1(g_{\tilde{B}}) = d_1(g) > 1$. 

This completes the proof of Proposition \ref{prop41}. 
\end{proof}

In what follows, let $f' : X' \to W$ be as in Proposition \ref{prop41} when $g \in {\rm Bir}\, (f) \setminus \Aut (f)$ and let $f' : X' \to W$ be $f : X \to B$ when $g \in \Aut (f)$. By Proposition \ref{prop41}, in both cases, $f' : X' \to W$ is an elliptically fibered Calabi-Yau threefold with $g' \in {\rm Bir}\, (f')$ and $g_W \in \Aut(W)$ with $d_1(g_W) >1$ for $g_W \in {\rm Bir}\, (W)$ induced from $g'$.

By Proposition \ref{prop23} applied for $f' : X' \to W$, we have an effective $\Q$-divisor $\Delta$ such that $-K_W$ is $\Q$-linearly equivalent to $\Delta$. We fix one of such $\Delta$. Let    
$$-K_W \sim_{\Q} \Delta = P + N$$
 be the Zariski decomposition of an effective $\Q$-divisor $\Delta$ in $N^1(W)_{\Q}$ such that $P$ is a nef $\Q$-divisor class and $N$ is an effective $\Q$-divisor class (See \cite[Theorem 7.4, Corollary 7.5]{Sa84} for the definition and basic properties of the Zariski decomposition. See also \cite{Ba09}). Note that the Zariski decomposition for an effective $\Q$-divisor $\Delta$ exists and it is also unique in the sense that if $\Delta = P'+D'$ is another Zariski decomposition of $\Delta$ in $N^1(W)_{\Q}$, then $P' = P$ in $N^1(W)_{\Q}$ and $D' = D$ not only as elements of $N^1(W)_{\Q}$ but also {\it as $\Q$-divisors} (See \cite[proof of Corollary 7.5]{Sa84}, which is based on the uniqueness in \cite[Theorem 7.4]{Sa84}). For this reason, we may and will identify the numerically equivalent class of $N$ with the unique effective $\Q$-divisor representing the class $N$, and we may thus speak of {\it the irreducible decomposition of} $N$:
$$N = \sum_{i=1}^{k} a_iN_i,$$
where $N_i$ are mutually different irreducible curves and $a_i$ are positive rational numbers. Note that the intersection matrix $((N_i.N_j)_W)_{i, j}$ is of negative definite unless $N = 0$ (\cite[Corollary 7.5]{Sa84}).

Following a request from the referee, we remark the next fact with proof:

\begin{lemma}\label{lem41} Under the notation above, if $h \in \Aut (W)$, then $h^*N = N$ as effective $\Q$-divisors (not only as elements of $N^1(W)_{\Q}$). 
\end{lemma}

\begin{proof} Note that in general $h^*\Delta \not= \Delta$ as $\Q$-divisors. However, since the Weil divisor $h^{*}(-K_W)$ is linearly equivalent to $-K_W$ and since $\Delta$ is $\Q$-linearly equivalent to $-K_W$, it follows that $h^{*}\Delta$ is $\Q$-linearly equivalent to $\Delta$. Thus $\Delta - h^*\Delta = 0$ in $N^1(W)_{\Q}$. Note also that $h^*N$ is effective as $N$ is effective and $h^*|_{N^1(W)_{\Q}}$ preserves the intersection form as $h \in \Aut (W)$. Therefore
$$\Delta = (h^{*}P + \Delta - h^*\Delta) + h^*N\,\, {\rm in}\,\, N^1(W)_{\Q}$$
is also the Zariski decomposition of $\Delta$. Thus $h^{*}N = N$ as $\Q$-divisors by the uniquness of the Zariski decomposition of an effective $\Q$-divisor $\Delta$ as explained above. This completes the proof.
\end{proof}

\begin{lemma}\label{lem42} Under the notation above, $P=0$ and $-K_W = N$ in $N^1(W)_{\Q}$.
\end{lemma}
\begin{proof} Since $g_W^*(-K_W) = -K_W$ in $N^1(W)_{\Q}$, it follows that $g_W^{*}P = P$ in $N^1(W)_{\Q}$ by Lemma \ref{lem41}. On the other hand, by Theorem \ref{thm21}, there is a non-zero nef $\R$-divisor class $D_W$ such that $g_W^*D_W = d_1(g_W)D_{W}$. 
Then
$$(P.D_W)_W = (g_W^*P.g_W^*D_W)_W = d_1(g_W)(P.D_W)_W.$$
Since $d_1(g_W) >1$ by Proposition \ref{lem41}, it follows that $(P.D_W)_W = 0$. Then nef classes $P$ and $D_W$ has to be proportional by the Hodge index theorem. On the other hand, since 
$$d_1(g_W) \not=1\,\, ,\,\, g_W^*D_W = d_1(g_W)D_W\,\, ,\,\, g_W^*P = P\,\, ,\,\, D_W \not= 0$$ 
in $N_1(W)_{\R}$, it follows that $P = 0$ in $N^1(W)_{\Q}$ and therefore $-K_W = N$ in $N^1(W)_{\Q}$. 
\end{proof}

Since klt is an open condition with respect to the coefficients of the boundary divisor, the pair $(W,E:= (1+\epsilon) N )$ with sufficiently small $\epsilon \in \Q_{>0}$ is also klt and 
$$K_W + E = \epsilon N = \sum_{i=1}^{k}\epsilon a_i N_i$$
in $N^1(W)_{\Q}$. Let us run the minimal model program with respect to $K_W+E$ (See eg. \cite[Theorem 3.47, Corollary 3.43 (1)]{KM98} for the minimal model program for normal projective surfaces with klt as singularities). Since the intersection matrix $((N_i.N_j)_W)_{i, j}$ is of negative definite, the output is not a Mori fiber space but the birational contraction $\tau : W \to B'$ such that 
$$\tau_*(\epsilon N) = 0$$ 
and under $\tau$, 
$$W \setminus {\rm Supp}\, N \simeq B' \setminus T,$$
where $T := \tau ({\rm Supp}\, N)$. Note that $T$ is a finite set of points by $\tau_*(\epsilon N) = 0$.   
In particular, $\tau_*(N) = 0$ and therefore 
$$K_{B'} = \tau_*K_W = \tau_*N = 0$$ 
in $N^1(B')_{\Q}$, as $K_{B'} = \tau_*K_W$ as Weil divisors. The composition $\tau \circ f' : X' \to B'$ is then an elliptically fibered Calabi-Yau threefold with numerically trivial $K_{B'}$. Hence this is an elliptically fibered Calabi-Yau threefold of Type $II_0$ by Proposition \ref{prop23}. 

Here we note that $\tau \circ f' : X' \to B'$ is nothing but $\tau \circ f : X \to B \to B'$ when $g \in \Aut (f)$. 

The next lemma will then complete the proof of Theorem \ref{mainthm2}.

\begin{lemma}\label{lem51} $g_{B'} : = \tau \circ g_{W} \circ \tau^{-1} \in 
\Aut (B')$ and $d_1(g_{B'}) > 1$.
\end{lemma}

\begin{proof} By Lemma \ref{lem41}, $(g_W^{\pm 1})^{*}(N) = N$ as $\Q$-divisors. Thus $g_W^{\pm 1}({\rm Supp} N) = {\rm Supp} N$. Hence $g_{B'}^{\pm 1}$ are isomorphisms on $B' \setminus T$ for $g_{B'} \in {\rm Bir}\, (B')$ induced from $g_W \in {\rm Bir}\, (W)$. In particular, $g_{B'} : B' \dasharrow B'$ is isomorphic in codimension one. Since $B'$ is a normal projective surface, it follows that $g_{B'} B' \to B'$ is an isomorphism, that is, $g_{B'} \in \Aut (B')$, by the Zariski main theorem. We also have 
$$d_1(g_{B'}) = d_1(g_W) = d_1(g_{\tilde{W}}) >1$$ 
by Propositions \ref{prop22} and  \ref{prop41}. This completes the proof of Lemma \ref{lem51}.
\end{proof} 

This completes the proof of Theorem \ref{mainthm2}.     


\end{document}